\documentclass{amsart}
\usepackage{amssymb,latexsym}
\usepackage{graphicx}
\usepackage{subfig}

\newtheorem{lemma}{Lemma}
\newtheorem{assertion}{Assertion}[lemma]

\theoremstyle{remark}

\def\Int{{\textnormal{Int}}}

\def\nbhd{neighborhood}

\def\R{{\mathbb R}}

\def\Cl{Cl}

\def\and{{\rm \ and\ }}

\begin{document}
\title[Morse Cells]{Morse Cells}
\author[H. King]{Henry C. King}
\address{Department of Mathematics\\
         University of Maryland\\
         College Park, MD 20742-4015}
\email{hck@math.umd.edu}
\urladdr{http://www.math.umd.edu/~hck/}

       
\date{\today}
\maketitle

\section{Introduction}
This paper is in the category {\em proof of something not deep or significant which I knew was true 
and told my students was true but 
for which I am not aware of a published reference}.
In particular, when encountering theorem 3.5 in Milnor's ``Morse Theory" which says a manifold
has the homotopy type of a CW complex with a $k$ cell for each index $k$ critical point of a Morse function,
I would remark that in fact as long as the gradientlike vector field is generic, the manifold
is a CW complex and the cells are the (closures of) the stable manifolds.
So to be honest I should write down a proof.
Doing so led to investigating  untico resolution towers, manifolds with nice corners on the boundary,
and other amusing notions.

Since posting version 1 of this paper others have kindly pointed out previous work on some of these things.
In particular the result above has appeared before, so I could have saved myself the trouble.
Indeed it appears there is enough activity that I was wrong to consider the question insignificant
as it appears to lead to much interesting work.
I stand corrected.
I'll leave this paper up though since I may end up using the work on resolution towers and stratifications elsewhere.
There is a nice discussion in http://mathoverflow.net/questions/86610
particularly the post by Liviu Nicolaescu which indicates work by Lizhen Qin see \cite{Q1} and \cite{Q2},
reference to \cite{BFK}, as well as his own \cite{N}.
There are a number of references to Laudenbach's appendix
to \cite{BZ} proving this result early on.
The post http://mathoverflow.net/questions/11375 is also relevant.
Many thanks to Patrick Massot for first alerting me to all this.

Manifolds with convex corners on the boundary were apparently developed in the early 1960s by Cerf
although I presume there was earlier work, particularly at the elementary level used here,
 since they arise whenever one takes the
product of manifolds with boundary.
More recently \cite{corners} is easily available online and gives references to earlier work.
Of course I should have included references to early work on Whitney stratifications, Thom stratifications
(I guess these are often called Thom-Mather stratifications but I learned them from Mather),
for example \cite{Thom}, \cite{GWPL}, and \cite{Mather}.
I have a vague recollection that Thom was originally thinking of an approach to stratifications
closer to the resolution towers given here, but later switched to the Thom data approach refined by Mather,
but this may be erroneous.
The resolution tower approach was mentioned at the end of \cite{KT}
and one motivation for this paper was to expand on those remarks.

\section{Manifolds with Convex Corners on the Boundary}

Usually corners of a smooth manifold are swept under the rug, smoothed out as soon as possible 
when they appear and viewed as a minor nuisance, not to be thought of too much.
However, while writing this paper, it became clear that it would be useful to use a certain type of corner on the boundary
of a smooth manifold and indeed to embrace this extra structure.

A smooth manifold with convex corners on the boundary is a manifold
with charts based on open subsets of $[0,\infty)^n$.
The corners give a stratification of the manifold according to the depth of the corner, a depth $k$
corner point in an $n$ dimensional smooth manifold with convex corners on the boundary
looks like $0$ in $[0,\infty)^k\times \R^{n-k}$.
We let $\partial_k M$ denote the set of points of a manifold $M$ with depth $k$
and $\bar{\partial}_k M = \Cl \partial_k M = \bigcup_{j\ge k}\partial_j M$.

\begin{lemma}\label{depth_defined}
	The depth of a point in a manifold with convex corners on the boundary is well defined.
\end{lemma}

\begin{proof}
	For any point $x$ in a manifold $M$ with convex corners on the boundary,
	define the tangent cone at $x$ to be the subset of the tangent space $T_xM$ of $M$ at $x$
	which are velocities at $x$ of a curve in $M$.  
	In other words, take a smooth curve $\alpha\colon ([0,\infty), 0)\to (M,x)$, then 
	$d\alpha(t)/dt|_{t=0}$ is in the tangent cone at $x$.
	Then note that $\dim M$ minus the depth of $x$ is the maximum dimension of a linear subspace
	of $T_xM$ which is contained in the tanget cone at $x$.
\end{proof}

If $N\subset M$ is a submanifold, we say that $N$ has boundary compatible with $M$ if at every point
of $N$, the pair $(M,N)$ is locally diffeomorphic to the pair $[0,\infty)^k\times (\R^{m-k},\R^{n-k})$,
where $k$ is the depth of the point in both $M$ and $N$.

If $h\colon M\to N$ is a smooth map between manifolds with convex corners on the boundary,
we say that $h$ is a strong submersion if it locally looks like projection of
$[0,\infty)^k\times \R^{n-k}\times [0,\infty)^\ell\times \R^{m-n-\ell}$ to $[0,\infty)^k\times \R^{n-k}$.

\begin{lemma}\label{submersion}
	Suppose $h\colon M\to N$ is a smooth map between manifolds with convex corners on the boundary.
	The following are equivalent.
	\begin{enumerate}
		\item $h$ is a strong submersion
		\item For any $x\in M$, suppose $x$ has depth $\ell$, $h(x)$ has depth $k$, and we choose local coordinates $g\colon (U,h(x))\to ([0,\infty)^k\times \R^{n-k},0)$
		in a \nbhd\ $U$ of $h(x)$.  Then 
		there are local coordinates 
		$$f\colon (V,x)\to ([0,\infty)^k\times \R^{n-k}\times [0,\infty)^{\ell-k}\times \R^{m-n-\ell+k},0)$$
		in a \nbhd\ $V$ of $x$ so that $ghf^{-1}$ is the restriction of projection 
		$$[0,\infty)^k\times \R^{n-k}\times [0,\infty)^{\ell-k}\times \R^{m-n-\ell+k}\to [0,\infty)^k\times \R^{n-k}.$$
	\end{enumerate}
\end{lemma}

\begin{proof}
	We need to show that 1 implies 2.
	Since $h$ is a strong submersion we know there are local coordinates 
	$f'\colon (V',x)\to ([0,\infty)^k\times \R^{n-k}\times [0,\infty)^{\ell-k}\times \R^{m-n-\ell+k},0)$ in a \nbhd\ $V'$ of $x$ and $g'\colon (U',h(x))\to ([0,\infty)^k\times \R^{n-k},0)$
	in a \nbhd\ $U'$ of $h(x)$ so that $g'hf'{}^{-1}(u,v,w,y)= (u,v)$.
	Define $f$ by $ff'{}^{-1}(u,v,w,y) = (u',v',w,y)$ where $(u',v')=gg'{}^{-1}(u,v)$.
	Then $$ghf^{-1}(u',v',w,y)= ghf'{}^{-1}(u,v,w,y)= gg'^{-1}(u,v) = (u',v').$$
\end{proof}

\begin{lemma}
	The composition of two strong submersions is a strong submersion.
\end{lemma}

\begin{proof}
	Suppose $h\colon M\to N$ and $h'\colon N\to N'$ are strong submersions.
	Pick any $x\in M$ and any local coordinates $g'\colon (V',h'h(x))\to ([0,\infty)^{k'}\times \R^{n'-k'},0)$.
	By Lemma \ref{submersion} there are local coordinates
	$$g\colon (V,h(x))\to ([0,\infty)^{k'}\times \R^{n'-k'}\times [0,\infty)^{k-k'}\times \R^{n-n'-k +k'},0)$$
	so that $g'h'g^{-1}$ is projection and there are local coordinates
	$$f\colon (U,x)\to ([0,\infty)^{k'}\times \R^{n'-k'}\times [0,\infty)^{k-k'}\times \R^{n-n'-k +k'}
	\times [0,\infty)^{\ell-k}\times \R^{m-n-\ell +k},0)$$
	so that $ghf^{-1}$ is projection.
	Then $g'h'hf^{-1} = g'h'g^{-1}ghf^{-1}$ is projection.
\end{proof}

\begin{lemma}\label{depth_submersion}
	Suppose $h\colon M\to N$ is a smooth map between manifolds with convex corners on the boundary
	and $h$ is a submersion, i.e., $dh$ has rank $\dim N$ everywhere.
	Suppose $h$ preserves depth, i.e., for each $k$, $h(\partial_k M)\subset \partial_k N$.
	Then $h$ is a strong submersion.
\end{lemma}

\begin{proof}
	Pick any $x\in M$ with depth $k$. If $k=0$ then $h$ is trivially a strong submersion at $x$, so suppose $k>0$.
	Choose any local coordinates 
		$f\colon (U,x)\to ([0,\infty)^{k}\times \R^{m-k},0)$ and
		$g\colon (V,h(x))\to ([0,\infty)^{k}\times \R^{n-k},0)$.
	After restricting $U$ we may as well assume that $f^{-1}(0\times (0,\infty)^{k-1}\times \R^{m-k})$
	is connected, for example if $f(U)$ is convex.
	Then after reordering coordinates we may suppose that 
	$$ghf^{-1}(0\times (0, \infty)^{k-1}\times \R^{m-k})\subset 0\times (0,\infty)^{k-1}\times \R^{n-k}$$
	since $h$ preserves the depth 1 points.
	By induction on $k$ we know that $ghf^{-1}$ restricts to a strong submersion from
	$f(U)\cap 0\times [0, \infty)^{k-1}\times \R^{m-k}$ to $0\times [0, \infty)^{k-1}\times \R^{n-k}$.
	So we may choose local coordinates $f'\colon (U',0)\to  [0,\infty)^{k-1}\times \R^{m-k}$
	in a \nbhd\ of $0$ in $0\times [0, \infty)^{k-1}\times \R^{m-k}$
	so that $ghf^{-1}f'{}^{-1}(y_2,y_3,\ldots,y_m) = (0,y_2,\ldots,y_n)$.
	So after composing $f$ with the map $y\mapsto (y_1,f'(y_2,\ldots,y_m))$ we may as well suppose that 
	$ghf^{-1}(0,y_2,\ldots,y_m) = (0,y_2,\ldots,y_n)$.	
	
	Let $h_i(y)$ denote the $i$-th coordinate of $ghf^{-1}(y)$.
	Since $dh$ has rank $n$ at $x$ we know $\partial h_1/\partial y_1(0)\ne 0$.
	Since $h$ preserves depth we know that for each $i\le k$, there is a $j_i\le k$ so that 
	for any $y$ near 0 with $y_i=0$ and $y_j\ne 0$ for $j\ne i$
	then $h_{j_i}(y)=0$. We know $j_1=1$.  We also know $j_i=i$ for $i>1$ since if we take $y$ near 0 with $y_1=y_i=0$ and $y_j\ne 0$
	for $j\ne 1,i$ then $h_j(y)=0$ iff $j=1$ or $i$ but then since $\partial h_1/\partial y_1\ne 0$
	if we make $y_1$ small and nonzero, only $h_i(y)=0$.
	
	Consider the change of coordinates $z_i = h_i(y)$ for $i\le n$ and $z_i=y_i$ for $i>n$.
	After incorporating this change of coordinates in $f$ we see $ghf^{-1}(y_1,y_2,\ldots,y_m) = (y_1,y_2,\ldots,y_n)$
	and thus $h$ is a strong submersion.
\end{proof}

\section{Stratified sets}

For this paper, a stratified set $X$ is a locally compact, second countable, Hausdorff, necessarily paracompact space 
which we denote $|X|$
and a decomposition of $|X|=\bigcup S$
into a locally finite union of disjoint subsets $S$ called strata, each of which has the
structure of a smooth manifold with convex corners on the boundary.
I feel more comfortable if we specify that the dimensions of the strata be bounded,
but perhaps this is not necessary.
We assume each stratum is locally closed in $|X|$.
If  $T$ and $S$ are disjoint strata and $T\cap \Cl S$ is nonempty we ask that $T\subset \Cl S$
and $\bar{\partial}_k T = T \cap \Cl \partial_k S$ for all $k$.

If $T\cap \Cl S$ is nonempty we say $T\preceq S$ and if in addition $T\ne S$ we say $T\prec S$.

The $k$ skeleton of a stratified set  $X$ is the union of its strata of dimension $\le k$.
This is also a stratified set.
The depth\footnote{There is also the notion of the depth of a stratified set,
	being the maximal difference in the dimensions of two strata of $X$.
	The only possible confusion I can think of is when $X$ is a single point in which case both notions of depth are zero.
	In any case we won't use this other notion of depth in this paper.}
 of a point in $|X|$ is the  depth of $x$ in the stratum containing $x$.
Note that the  depth $k$ points of $X$ inherit the structure of a stratified set without boundary.

If $N$ is a smooth manifold and $f\colon |X|\to N$ we say $f$ is smooth if the restriction of $f$ to each stratum 
of $X$ is $C^\infty$.

If $U$ is an open subset of $|X|$ then $U$ inherits the structure of a stratified set from $X$, the strata are $S\cap U$
for strata $S$ of $X$ which intersect $U$.

I fully expect that somewhere below I will inadvertently write $X$ where I should write $|X|$ and apologize in advance for the imprecision.

\begin{lemma}\label{compact-finite}
	If $K$ is a compact subset of a stratified set there are only finitely many strata $S$ so that
	$\Cl S\cap K$ is nonempty.
\end{lemma}

\begin{proof}
	By local finiteness, for each $x\in K$ there is an open \nbhd\ $V_x$ of $x$ which intersects only finitely
	many strata.  The open cover $\{V_x\}$ of $K$ has a finite subcover.
\end{proof}

\begin{lemma}\label{compact_exhaustion}
	If $X$ is a  stratified set, there is a sequence of compact subsets $K_i\subset |X|$, $i=1,2,\ldots$
	so that $\bigcup_{i=1}^\infty K_i = |X|$ and for each $i$, $K_i\subset \Int K_{i+1}$.
\end{lemma}

\begin{proof}
	This only requires second countability and local compactness.
	Let $\{L_i\}$ be a countable basis for the topology of $|X|$.
	By local compactness, every $x\in |X|$ is contained in some $L_i$ whose closure is compact.
	Consequently by taking a subsequence, we have a countable collection of open sets $\{L'_i\}$
	so that each $L_i'$ has compact closure and $\bigcup_{i=1}^\infty L_i' = |X|$.
	Suppose by induction we have a sequence of compact sets $K_1,\ldots,K_{k-1}$ so that for each $i<k-1$,
	$K_i\subset \Int K_{i+1}$ and so that for each $i<k$, $L'_i\subset K_i$.
	We must construct $K_k$. By compactness of $K_{k-1}$ there is a finite collection of the $L_i'$ which cover
	$K_{k-1}$, let $K_k$ be the union of the closures of these $L_i'$ as well as $\Cl L'_k$.
\end{proof}

\section{Thom stratified sets}

A stratified set $X$ is a Thom stratified set if it is equipped with Thom data $\{(U_S,\pi_S,\rho_S)\}$.
Thom data is for every stratum $S$  an open neighborhood $U_S$ of $S$ in $|X|$,
a smooth retraction $\pi_S\colon U_S\to S$ and a smooth distance function $\rho_S\colon U_S\to [0,\infty)$
so that:
\begin{itemize}
	\item For every stratum $S$, $S = \rho_S^{-1}(0)$.
	\item If $T\prec S$ then $\pi_T = \pi_T\pi_S$ and $\rho_T = \rho_T\pi_S$ wherever both sides are defined.
	\item If $T\prec S$ then  $\rho_T\times \pi_T|\colon U_T\cap S\to (0,\infty)\times T$ 
	is a strong submersion.
	\item If $T\prec S$ and $x\in S\cap U_T$ the depth of $x$ in $S$ equals the depth of $\pi_T(x)$
	in $T$.
\end{itemize}

Suppose we have a Thom stratified set and for every stratum $S$ we choose an open \nbhd\ $U'_S$
of $S$ in $U_S$.
Then we obtain another Thom stratified set by restricting each $\pi_S$ and $\rho_S$ to $U'_S$.
We do this often and refer to it as shrinking the $U_S$.
Of course we could avoid this by changing the definition so that $\pi_S$ and $\rho_S$ are germs at $S$,
and this is usually done by other authors.
I have chosen not to do so.

We say that a smooth function $\delta_S\colon S\to (0,\infty)$ is a frontier limit function if the restriction of $\pi_S$
to $\{x\in U_S\mid \rho_S(x)\le \delta_S\pi_S(x)  \}$ is proper,
i.e., for any compact $K\subset S$, $\pi_S^{-1}(K)\cap \{x\in U_S\mid \rho_S(x)\le \delta_S\pi_S(x)  \}$ is compact.
A frontier limit function may not exist, for example $\rho_S$ might approach 0 as you approach the frontier of $U_S$.
However, after shrinking $U_S$ they are guaranteed to exist.

\begin{lemma}\label{frontier}
	For any stratum $S$ of a Thom stratified set, after shrinking $U_S$ there is some frontier limit function.
\end{lemma}

\begin{proof}
	Probably the reader can come up with a nicer proof than the following, but here goes.
	For any $x\in S$, choose a compact \nbhd\ $V_x$ of $x$ in $U_S$.
	By paracompactness we may choose a locally finite refinement $\{O_\alpha\}$ of the cover $\{S\cap\Int V_x\}$ of $S$.
	For each $\alpha$ there is an $x_\alpha\in S$ so that $O_\alpha\subset V_{x_\alpha}$.
	Let $U'_S = \bigcup_\alpha \pi_S^{-1}(O_\alpha)\cap \Int V_{x_\alpha}$
	and let $U''_S = U_S\cap \Cl U'_S$.
	We will shrink $U_S$ to $U'_S$.
	
	I claim that $\pi_S|\colon U''_S\to S$ is proper.
	Take any compact $K\subset S$, then by local finiteness $K$ intersects only a finite number of $O_\alpha$
	and hence $\pi_S^{-1}(K)\cap U''_S$ is covered by a finite number of the compact $V_{x_\alpha}$
	and is hence compact.
	
	For each $x\in S$ choose $\delta_x>0$ less than the minimum of $\rho_S$
	on the compact set $(U''_S-U'_S)\cap \pi_S^{-1}(V_x)$.
	Take a smooth partition  of unity $\{\phi_x\}$ for the open cover $\{S\cap \Int V_x\}$ of $S$
	and define $\delta_S(y) = \Sigma_{x\in S}\phi_x(y)\delta_x$.
	Note $\{x\in U''_S\mid \rho_S(x)\le \delta_S\pi_S(x)  \}\subset U'_S$.
\end{proof}

\begin{lemma}\label{frontier_proper}
	Suppose $\delta_S$ is a frontier limit function.
	Then the restriction of $\pi_S\times\rho_S$ to $\{x\in U_S\mid \rho_S(x)< \delta_S\pi_S(x)  \}$
	is a proper map to $\{(x,t)\in S\times (0,\infty) \mid  t < \delta_S(x) \}$
\end{lemma}

\begin{proof}
	Let $K$ be any compact subset of $W = \{(x,t)\in S\times (0,\infty) \mid  t < \delta_S(x) \}$
	and let $K'$ be the projection of $K$ to $S$.
	By definition we know $K'' = \pi_S^{-1}(K')\cap \{(x,t)\in S\times (0,\infty) \mid  t \le \delta_S(x) \}$
	is compact.  But $(\pi_S\times\rho_S)^{-1}(K)\cap \{x\in U_S\mid \rho_S(x)< \delta_S\pi_S(x)  \}$ is a closed subset of $K''$ and hence compact.
\end{proof}

The following is really just a version of part of Lemma 17 of \cite{Tico}
which was stated in the Whitney stratified context.
We rehash the proof given there which simplifies a bit since we are not simultaneously constructing the $\pi_S$.

\begin{lemma}\label{goodthom}
	Let $X$ be a Thom stratified set 
	with Thom data $\{(U_S,\pi_S,\rho_S) \}$.
	Then after perhaps shrinking the $U_S$ and restricting the $\rho_S$ and $\pi_S$
	we may suppose that:
	\begin{enumerate}
		\item If $T\prec S$ then $U_S\cap U_T= \pi_S^{-1}(S\cap U_T)$.
		\item If $T\prec S$ then $\pi_T\pi_S = \pi_T$ and
		$\rho_T\pi_S = \rho_T$
		on $U_S\cap U_T$.
		\item If $U_S\cap U_T$ is nonempty, then either $S\prec T$,
		$T\prec S$, or $S=T$.
		\item If $T_1\prec T_2\prec\cdots\prec T_\ell\prec S$  then
		$$(\rho_{T_1},\rho_{T_2},\ldots,\rho_{T_\ell},\pi_{T_1})
		\colon S\cap U_{T_1}\cap\cdots\cap U_{T_\ell} \to
		(0,\infty)^\ell\times T_1$$
		is a strong submersion.
		\item $\pi_S\times \rho_S\colon U_S\to W_S$ is a proper map onto $W_S$, where
		$W_S=\{(x,t)\in S\times (0,\infty) \mid  t < \gamma_S(x) \}$
		for some smooth $\gamma_S\colon S\to (0,\infty)$.
		\item Any compact subset of $|X|$ intersects only a finite number of the $U_S$.
	\end{enumerate}
\end{lemma}

\begin{proof}
	By Lemma \ref{compact_exhaustion} there is a sequence of compact subsets
	$K_1, K_2,\ldots$  so that $K_i\subset \Int K_{i+1}$ for all $i$
	and $\bigcup_{i=1}^\infty  K_i = |X|$.
	For each stratum $T$ which does not intersect $K_1$, let $j$ be the largest index so that $T\cap K_j$ is empty.
	Replace $U_T$ by the smaller open \nbhd\ $U_T-K_j$.
	Now any compact set $K$ is contained in some $K_i$. If $K$ intersects some $U_S$ then $K_i$ intersects $U_S$
	so $K_i$ intersects $S$, but by local finiteness $K_i$ only intersects a finite number of strata
	so the sixth condition holds.

	The third condition is easily obtained by shrinking the $U_S$.
	We suppose by induction on $i$ that if $S$ and $T$ are strata which intersect $K_i$ and neither $T\prec S$
	nor $S\preceq T$ then $U_S\cap U_T$ is empty.
	The inductive step shrinks the $U_S$ leaving $U_S\cap K_i$ fixed.
	In the end, each $U_S$ will still be open even though it may have been shrunk infinitely often.
	
	Note that $U_S\cap \Cl S = S$ since $\pi_S$ is a retraction.
	For each stratum $S$ we may pick a \nbhd\ $U''_S$ of $S$ in $U_S$ so that $\Cl U''_S - U_S = \Cl S -S$
	since $S$ and $|X|-(U_S \cup \Cl S)$ are disjoint closed subsets of the normal space $|X| - (\Cl S - S)$.
	We suppose by induction on $k$ that for all strata $S$ of dimension $< k$ we have chosen a 
	frontier limit function $\gamma_S$ so that if  $U'_S = \{x\in U_S\mid \rho_S(x)< \gamma_S\pi_S(x)  \}$ then
	$U'_S\subset U''_S$ and 
	for any $T\prec S$ we know $\pi_S(U'_T\cap U'_S)\subset U_T$ 
	and consequently $\pi_T\pi_S=\pi_T$ and $\rho_T\pi_S=\rho_T$ on $U'_S\cap U'_T$.
	For the inductive step, let $S$ be a stratum of dimension $k$.
	We will need to find an appropriate $\gamma_S$.
	For each $x\in S$ pick a \nbhd\ $W_x $ of $x$ in $U''_S$ with compact closure.
	Then $W_x$ intersects only finitely many $U_T$.
	Also, if $T\prec S$ and $W_x$ intersects $\Cl U'_T$ then $W_x$ intersects $\Cl U''_T$ so $W_x$ intersects $U_T$
	and thus this occurs for only finitely many $T$.	
	Let $$W'_x=W_x\cap \bigcap_{  x\in U_T}\pi_S^{-1}(S\cap U_T)-\bigcup_{ x\not\in \Cl U'_T, T\prec S} \Cl U'_T.$$
	By Lemma \ref{frontier} we may choose $\gamma_S$ so that $U'_S\subset \bigcup_{x\in S} W'_x$.
	Note that $\pi_S(U'_T\cap U'_S)\subset U_T$ for all $T\prec S$.
	(If $y\in U'_T\cap U'_S$ then $y\in W'_x$ for some $x$.
	If $x\in U_T$ then $y\in \pi_S^{-1}(S\cap U_T)$ so $\pi_S(y)\in U_T$.
	If $x\not\in U_T$ then $x\not\in \Cl U'_T$ so $y\not\in \Cl U'_T$, contradiction.)
	So by induction we may choose $\gamma_S$ for all strata.
	
	At this point, I claim we may shrink all $U_S$ to $U'_S$ and the conclusions of this Lemma hold.
	Let us see why.
	For the first conclusion we must show that if $T\prec S$ then $U'_S\cap U'_T= \pi_S|_{U'_S}^{-1}(S\cap U'_T)$,
	i.e., $U'_S\cap U'_T= \pi_S^{-1}(S\cap U'_T)\cap U'_S$.
	
	Suppose $x\in \pi_S^{-1}(S\cap U'_T)\cap U'_S-U'_T$ for some strata $T\prec S$.
	Let $P$ be the stratum containing $x$.
	Since $\rho_S\times \pi_S$ submerses $P\cap U'_S$ we may integrate an appropriate vector field on $P\cap U_S'$ to
	obtain a flow $\phi_t(y)$ on $P\cap U'_S$ so that $\pi_S\phi_t(y)=\pi_S(y)$ and $\rho_S\phi_t(y) = \rho_S(y)+t$.
	For any $y$, this flow is defined for all $t\in (-\rho_S(y),\gamma_S\pi_S(y)-\rho_S(y))$
	and $\lim_{t\to -\rho_S(y)}\phi_t(y) = \pi_S(y)$.
	Define a continuous curve $\beta\colon [0,\rho_S(x)]\to U'_S$ by $\beta(t)=\phi_{t-\rho_S(x)}(x)$ for $t>0$ 
	and $\beta(0)=\pi_S(x)$.
	Let $t_0 = \inf\{t\mid \beta(t)\not\in U'_T \}$.
	Since $U'_T$ is open we know $t_0 > 0$.
	For $t<t_0$ we know 
	$\pi_T\beta(t) = \pi_T\pi_S\beta(t) = \pi_T\pi_S(x)$
	and $\rho_T\beta(t)=\rho_T\pi_S\beta(t)=\rho_T\pi_S(x)$ are constants.
	But by properness of $\rho_T\times \pi_T$ on $U'_T$ we know $\{y\in U'_T\mid \pi_T(y)  = \pi_T\pi_S(x)
	{\rm \ and\ } \rho_T(y) = \rho_T\pi_S(x)\}$ is compact so $\beta(t_0)\in U'_T$
	so $\beta(t)\in U'_T$ for $t$ slightly larger than $t_0$, a 
	contradiction.  So we know $\pi_S^{-1}(S\cap U'_T)\cap U'_S\subset U'_S\cap U'_T$.
	
	Now suppose $x\in U'_S\cap U'_T$ but $\pi_S(x)\not\in U'_T$ for $T\prec S$.
	As above we may find a continuous $\beta\colon [0,\rho_S(x)]\to U'_S$ so that $\pi_S\beta(t)=\pi_S(x)$,
	$\rho_S\beta(t)=t$, and $\beta(\rho_S(x))=x$.
	Let $t_0 = \sup \{ t \mid \beta(t)\not\in U'_T  \}$.
	For $t>t_0$ we know as above that $\pi_T\beta(t)$ and $\rho_T\beta(t)$ are constant so $\beta(t_0)\in U'_T$,
	a contradiction. So $U'_S\cap U'_T= \pi_S^{-1}(S\cap U'_T)\cap U'_S$.

	The second condition is immediate and we already shrunk to satisfy the third condition.
	For the fourth condition
	we must show that if $T_1\prec T_2\prec\cdots\prec T_\ell\prec S$  then
	$(\rho_1,\rho_2,\ldots,\rho_\ell,\pi_{T_1})
	\colon S\cap U'_{T_1}\cap\cdots\cap U'_{T_\ell} \to
	(0,\infty)^\ell\times T_1$
	is a strong submersion.
	Note this map is the composition of the strong submersions
	$\rho_{T_\ell}\times \pi_{T_\ell}\colon S\cap U'_{T_1}\cap\cdots\cap U'_{T_{\ell}} \to 	(0,\infty)\times T_\ell$
	and $id\times(\rho_{T_1},\rho_{T_2},\ldots,\rho_{T_{\ell-1}},\pi_{T_1})$.
	The fifth condition is a consequence of Lemma \ref{frontier_proper} and we already shrunk to satisfy the sixth.	
\end{proof}

We say $\{ (U_S,\pi_S,\rho_S,\gamma_S)  \}$ is enhanced Thom data for $X$ if it satisfies the conclusions of Lemma \ref{goodthom}.

It is often convenient to reparameterize the distance functions $\rho_S$.

\begin{lemma}\label{reparam}
	Suppose  $\{ (U_S,\pi_S,\rho_S,\gamma_S)  \}$ is enhanced Thom data for a Thom stratified set $X$.
	For each stratum $S$, suppose we pick a smooth $\gamma'_S\colon S\to (0,\infty)$
	and a smooth parameterized family of diffeomorphisms $\kappa_{S,y}\colon [0,\gamma_S(y))\to [0,\gamma'_S(y))$
	for $y\in S$.  Define $\rho'_S(x) = \kappa_{S,\pi_S(x)}\rho_S(x)$.
	Then $\{(U_S,\pi_S,\rho'_S,\gamma'_S)\}$ is enhanced Thom data for $X$.
	
	In particular, if we set $\gamma'_S = 1$ everywhere and let $\kappa_{S,y}(t) = t/\gamma_S(y)$ we see that
	any Thom stratified set has enhanced Thom data with 
	all $\gamma_S$ the constant 1.
\end{lemma}

\begin{proof}
	Conditions 1, 3, and 6 in Lemma \ref{goodthom} are immediate since $U_S$ is unchanged.
	By definition of smooth parameterization, the map $$\kappa_S\colon 
	\{(x,t)\in S\times (0,\infty) \mid  t < \gamma_S(x) \}\to \{(x,t)\in S\times (0,\infty) \mid  t < \gamma'_S(x) \}$$
	given by $\kappa_S(x,t)=(x,\kappa_{S,x}(t))$ is smooth and by the inverse function theorem
	is a diffeomorphism, so condition 5 holds.
	Condition 2 follows since $$\rho'_T\pi_S(x) = \kappa_{T,\pi_T(\pi_S(x))}\rho_T(\pi_S(x)) =  
	\kappa_{T,\pi_T(x)}\rho_T(x) = \rho'_T(x).$$
	Condition 4 with $\ell=1$ follows since $(\pi_{T_1},\rho'_{T_1}) = \kappa_{T_1}(\pi_{T_1},\rho_{T_1})$
	and is hence a strong submersion.  
	Condition 4 for all $\ell$ follows by induction as in the proof of Lemma \ref{goodthom}.
\end{proof}

\begin{lemma}\label{proper_r}
	Suppose $\{ (U_S,\pi_S,\rho_S,\gamma_S)  \}$ is enhanced Thom data for $X$.
	Suppose also that we have for each stratum $S$ of $X$ a smooth $\delta_S\colon S\to \R$ so that $\gamma_S(x)>\delta_S(x)>0$
	for all $x\in S$.
	Then there is a smooth proper map $r\colon |X|\to \R$ so that $r\pi_S(x)=r(x)$ whenever $\rho_S(x)\le \delta_S\pi_S(x)$.
\end{lemma}

\begin{proof}
	After reparameterizing using Lemma \ref{reparam} we may as well suppose that $\gamma_S = 1$ and $\delta_S = 1/2$ everywhere.
	For example use the reparameterization 
	$$\kappa_{S,y}(t)  = (\gamma_S(y)-\delta_S(y))t/((\gamma_S(y)-2\delta_S(y))t + \delta_S(y)\gamma_S(y)).$$
	
	We will find below a sequence of smooth functions $r_i\colon |X|\to [0,1]$ so that:
	\begin{itemize}
		\item For each $i$, $r_i^{-1}(0)$ is compact.
		\item If $j>i$ then $r_i^{-1}([0,1))\subset r_j^{-1}(0)$.
		\item For any $x\in |X|$ there is an $i$ so that $r_i(x)<1$.
		\item For each $i$ and stratum $S$, if $x\in \rho_S^{-1}([0,1/2])$ then $r_i(x)=r_i\pi_S(x)$.
	\end{itemize}
	Given these $r_i$ we just let $r = \Sigma r_i$.  Note $r$ is well defined, for any $x\in |X|$, pick $i$ so $r_i(x)<1$.
	Then $r_j(y)=0$ for all $j>i$ and $y$ near $x$, so $r(y) = \Sigma_{j=1}^i r_j(y)$.  In fact if $r_i(x)\in (0,1)$
	then $r(y) = i-1+r_i(y)$ for $y$ near $x$.
	To see $r$ is proper, note that $r^{-1}([0,n]) \subset r_{n+1}^{-1}(0)$ which is compact.
	
	So let us find the $r_i$.
	For each $x\in |X|$ pick a smooth function $q_x\colon |X|\to [0,1]$ as follows.
	Let $S$ be the lowest dimensional stratum so that $x\in U_S$ and $\rho_S(x)\le 1/2$.
	Choose a smooth function $p\colon S\to [0,1]$ with compact support so that:
	\begin{itemize}
		\item For all $y$ in some \nbhd\ of $\pi_S(x)$ in $S$, $p(y)=1$.
		\item If $T\prec S$ and $y\in S\cap \rho_T^{-1}([0,1/2])$ then $p(y)=0$.
	\end{itemize}
	Now choose some smooth $\alpha\colon [0,1]\to [0,1]$ so that 
	$\alpha(t)=1$ for $t< 3/5$ and $\alpha(t)=0$ for $t>4/5$.
	We then define $q_x(y)= p\pi_S(y) \alpha\rho_S(y)$ for $y\in U_S$ and $q_x(y)=0$ for $y\not\in U_S$.
	Note that $q_x$ has compact support and is 1 on a \nbhd\ of $x$.
	Also for any $T$, if $y\in \rho_T^{-1}([0,1/2])$ then $q_x\pi_T(y)=q_x(y)$.
	This is trivial if $y\not\in U_S$ or $T\prec S$ since both sides are 0.
	If $y\in U_S$ and $T=S$ then
	$q_x\pi_S(y) = p\pi_S(y) = q_x(y)$
	and if $y\in U_S$ and $S\prec T$ then 
	$$q_x\pi_T(y) = p\pi_S\pi_T(y)\alpha\rho_S\pi_T(y) = p\pi_S(y)\alpha\rho_S(y) = q_x(y).$$
	Let $V_x$ denote the interior of $q_x^{-1}(1)$ which is an open \nbhd\ of $x$.
	
	By Lemma \ref{compact_exhaustion} there is  a countable sequence of compact subsets $K_1\subset K_2 \subset\cdots$
	so that $\bigcup_{i=1}^\infty K_i = |X|$.
	Suppose now we have constructed $r_1,\ldots,r_{i-1}$ so that:
	\begin{itemize}
		\item For each $j<i$, $1-r_j$ has compact support.
		\item If $k<j<i$ then $r_k^{-1}([0,1))\subset r_j^{-1}(0)$.
		\item For each $j < i$, $K_j\subset r_j^{-1}(0)$.
		\item For each $j<i$ and stratum $S$, if $x\in \rho_S^{-1}([0,1/2])$ then $r_j(x)=r_j\pi_S(x)$.
	\end{itemize}
	By compactness of $K_i\cup \Cl r_{i-1}^{-1}([0,1))$ we may choose a finite number of points $z_1,\ldots, z_k$
	so that$K_i \cup \Cl r_{i-1}^{-1}([0,1))\subset \bigcup_{j=1}^k V_{z_j}$.
	We define $r_i(y)=\Pi_{j=1}^k (1-q_{z_j}(y))$.
	Note the support of $1-r_i$ is contained in the union of the supports of the $q_{z_j}$ and is hence compact.
	Note that $r_i$ is 0 on any $V_{z_j}$ and hence $K_i\cup r_{i-1}^{-1}([0,1))\subset r_i^{-1}(0)$ and
	consequently $r_j^{-1}([0,1))\subset r_i^{-1}(0)$ for all $j<i$.
	So we can continue the construction for all $i$.
	
	It only remains to show the $r_i$ have the required properties.
	We know $r_i^{-1}(0)$ is compact since it is a closed subset of the support of $1-r_{i}$ which is compact.
	Any $x\in |X|$ is contained in some $K_i$ and so $r_i(x)=0$.
\end{proof}

\section{Untico Resolution Towers for Thom stratified sets}

Suppose we have a tico $A$ in a manifold $M$, i.e., $A$ is an immersed codimension one submanifold of $M$
in general position with itself. We could split $M$ along $A$ and obtain a manifold with convex corners.
With this analogy in mind we define an untico in a manifold $Z$ with convex corners to be the closure of a union of
connected components of $\partial_1 Z$.

An untico resolution tower is analogous to a resolution tower as defined in \cite{AK}, except that
instead of ticos in manifolds we have unticos in manifolds with convex corners.
So we have a partially ordered index set which we may as well take to be the strata of a  stratified set $X$
with the partial order $\prec$ and for each index $S$ a manifold $V_S$ with convex corners on the boundary and for each 
$T\prec S$ an untico $V_{TS}\subset V_S$ and a proper map $p_{TS}\colon V_{TS} \to V_T$ so that:
\begin{enumerate}
	\item If $V_{TS}\cap V_{PS}$ is nonempty then either $T\prec P$, or $P\prec T$, or $P=T$.
	Moreover, $p_{TS}(V_{TS}\cap V_{PS})\subset V_{PT}$ if $P\prec T\prec S$.
	\item If $P\prec T\prec S$ then $p_{PT}p_{TS}(x) = p_{PS}(x)$ for all $x\in V_{TS}\cap V_{PS}$.
	\item If $P\preceq T\prec S$ then $p_{TS}^{-1}(\bigcup_{Q\prec P} V_{QT}) = \bigcup _{Q\prec P} V_{QS}\cap V_{TS}$.
	\item Any depth 1 point of $V_S$ is in at most one $V_{TS}$.
	\item If $T\prec S$ then $\Cl p_{TS}^{-1}( \partial_k V_T - \bigcup_{P\prec T}V_{PT}) = \Cl( \partial_k V_S -\bigcup_{P\prec S}V_{PS})\cap V_{TS}$.
	\item The index set is countable and for each $T$ there are only finitely many $S$ so that $T\prec S$.
\end{enumerate}

Conditions 1 through 5 are the analogues of conditions I through V in the definition of resolution tower in \cite{AK}.
The last condition 6 does not appear in the definition of resolution tower in \cite{AK} but probably should have
since it is essential for the realization to be locally compact and second countable.

As in \cite{AK} the realization of an untico resolution tower is the quotient space $\bigcup V_S/\sim$ where $\sim$
is the equivalence relation generated by $x\sim p_{TS}(x)$ for $x\in V_{TS}$.

\begin{lemma}
	The realization of an untico resolution tower $\{V_S,V_{TS},p_{TS}\}$
	is a stratified set with strata $V_S-\bigcup_{T\prec S}V_{TS}$.
\end{lemma}

\begin{proof}
	For each $S$, let $V'_S = V_S -\bigcup_{T\prec S}V_{TS}$.	
		Let $Z$ denote the realization, let $q\colon \bigcup V_S\to Z$ be the quotient map
		and let $q_S=q|_{V_S}$. Pick any $z\in Z$. Then $z=q(y)$ for some $y\in V_S$.
		We may as well choose such a $y$ so that $S$ has the smallest dimension possible.
		We cannot have $y\in V_{TS}$ for any $T$ since that would mean $z=q(y')$ where $y' = p_{TS}(y)\in V_T$
		and $T$ has smaller dimension than $S$.  So $y\in V'_S$.
		
	Suppose $y\sim x$ for some $y\in V'_S$ and $x\in V'_P$
	and $y\ne x$.
	There must be some points $y_i\in V_{S_i}$, $i=0,\ldots,k$ so that $y=y_0$, $x = y_k$ and 
	for each $i=1,\ldots, k$ either $y_i = p_{S_iS_{i-1}}(y_{i-1})$ or $y_{i-1} = p_{S_{i-1}S_i}(y_i)$.
	Take such a sequence where $k$ is as small as possible.
	We cannot have $y_1 = p_{S_1S_{0}}(y_{0})$ since $y\not\in V_{S_1S}$, so $y = p_{SS_1}(y_1)$.
	We cannot have  $y_1 = p_{S_1S_{2}}(y_{2})$
	since then $y = p_{SS_{2}}(y_{2})$ and a shorter sequence is possible.
	So  $y_2 = p_{S_2S_{1}}(y_{1})$ and thus $y_1\in V_{S_2S_1}$ and so $y = p_{SS_1}(y_1)\in V_{S_2S}$, a contradiction.
	
	So we know $Z$ can be decomposed into the disjoint manifolds $V'_S$.
	We must show all the conditions for this decomposition to be a stratified set.
	The reader can no doubt simplify the argument I give below which seems more involved than it should be.
	
	If $T\preceq S$ we define the height of $S$ above $T$ as the maximal $k$ so that there are
	$S_0,\ldots,S_k$ with $T=S_0$, $S=S_k$, and $S_0\prec S_1\prec\cdots\prec S_k$.
	
	Let us show $Z$ is locally compact.  So take any $z\in Z$, then $z=q_T(x)$ with $x\in V'_T$ for some $T$.
	Take an open \nbhd\ $W_T$ of $x$ in $V'_T$ so that $\Cl W_T$ is a compact subset of $V'_T$.
	Suppose by induction we have found open sets $W_S$ in $V_S$ for all strata $S$ of height $<k$ above $T$ so that
	$\Cl W_S$ is compact and
	if $ Q\prec S$ then $p_{QS}^{-1}(W_Q) = W_S\cap V_{QS}$ if $T\preceq Q$ and $ W_S\cap V_{QS}$ is empty otherwise.
	Take any stratum $S$ of height $k$ above $T$.
	Let $C_S = \bigcup_{ T\preceq Q\prec S}  p_{QS}^{-1}( W_Q)$.
	Since each $p_{QS}$ is proper and there are only finitely many $Q$ above $T$ we know that 
	$\Cl C_S$ is compact.
	Suppose $x\in C_S\cap V_{QS} $ for some $Q$ with $ Q\prec S$
	and $x\not\in  p_{QS}^{-1}( W_Q)$ if $T\preceq Q$.
	Then $x\in  p_{PS}^{-1}( W_P)$ for some $P\ne Q$ with  $T\preceq P\prec S$.
	Since $x\in V_{PS}\cap V_{QS}$ we know either $P\prec Q$ or $Q\prec P$.
	If $P\prec Q$ we know $p_{QS}(x)\in p_{PQ}^{-1}(W_P)$ so $p_{QS}(x)\in W_Q$ a contradiction.
	If $Q\prec P$ we know $p_{PS}(x)\in W_P\cap V_{QP}$ so $T\preceq Q$ and $p_{PS}(x)\in p_{QP}^{-1}(W_Q)$
	and thus $p_{QS}(x)\in W_Q$, a contradiction.
	So $C_S\cap V_{QS} =  p_{QS}^{-1}( W_Q)$ for all $T\preceq Q\prec S$
	and $C_S\cap V_{QS}$ is empty otherwise.	
	Take a compact \nbhd\ $W'_S$ of $\Cl C_S$ in $V_S$. Now let $W_S$ be an open \nbhd\ of $C_S$ in $\Int W'_S$
	 so that $W_S \cap V_{QS} = C_S\cap V_{QS}$ for all $Q\prec S$.
	So by induction we know we can find $W_S$ for all $T\preceq S$.
	Then $\bigcup_{T\preceq S }	q_S(W_S)$ is an open \nbhd\ of $z$ in $Z$.
	Its closure is compact since it is contained in the compact set $\bigcup_{T\preceq S }	q_S(\Cl W_S)$.
	So $Z$ is locally compact.
	
	Let us now show $Z$ is Hausdorff.
	Take $z\ne z'$ in $Z$.
	We have $x\in V'_T$ and $x'\in V'_{T'}$ for some $T$ and $T'$ so that $z=q_T(x)$ and $z' = q_{T'}(x')$.
	Just as above we construct $W_S$ for all $T\preceq S$ and $W'_S$ for all $T'\preceq S$.
	We must just make sure that each $W_S\cap W'_S$ is empty.

	Now let us show second countability.
	First put a  metric $d_S$ on each $V_S$.
	For each $T$ we can take a countable basis $\{W_{Ti}\}$ of $V'_T$ with each $\Cl W_{Ti}$ a compact subset of $V'_T$.
	We introduce another integer index $j$.  
	As above, for each $S$ with $T\prec S$ and each $i,j$ we construct an open subset $W_{TSij}$ of $V_S$ with compact closure,
	starting with $W_{TTij}=W_{Ti}$.
	The only extra bit is we make sure that $W_{TSij}$ lies within distance $1/j$ of 
	$C_{TSij}=\bigcup_{ T\preceq Q\prec S}  p_{QS}^{-1}( W_{TQij})$
	and if $j>1$ we require that $W_{TSij}\subset W_{TSij-1}$.
	We then claim the countable collection of open sets $U_{Tij} = \bigcup_{T\preceq S} q_S(W_{TSij})$
	is a basis for the topology of $Z$.
	So take any open set $O\subset Z$ and any $z\in O$.
	Suppose $z=q_T(x)$ for $x\in V'_T$.
	Start out by choosing $i$ so that $x\in W_{Ti}$ and $\Cl W_{Ti}\subset q_T^{-1}(O)$.
	We suppose by induction on $k$ that we have a sequence of integers
	$j_0\le j_1 \le \cdots \le j_{k}$ so that for any $S$ of height $\ell\le k$ above $T$ we have $\Cl W_{TSij}\subset q_S^{-1}(O)$
	for all $j\ge j_\ell$.
	To complete the induction, note that if $S$ has height $k+1$ above $T$ then
	$\Cl C_{TSij_k}\subset q_S^{-1}(O)$ since if $T\prec Q\prec S$ then 
	$\Cl  p_{QS}^{-1}( W_{TQij_k}) \subset p_{QS}^{-1}q_Q^{-1}(O) \subset q_S^{-1}(O)$.
	So we need only choose $j_{k+1}$ so that for each $S$ of height $k+1$ above $T$,
	the set of points of distance $\le 1/j_{k+1}$ from $C_{TSij_k}$ is contained in $q_S^{-1}(O)$.
	Finally, if $k$ is the maximum height above $T$ we know that $x\in U_{Tij_k}\subset O$.
\end{proof}
  
I presume the realization of an untico resolution tower is a Thom stratified set,
if perhaps one puts a few simple restrictions on the tower, but haven't bothered to think about it
to discover what restrictions, if any, are needed.
However we will show that any Thom stratified set is the realization of an untico resolution tower.
The construction is quite simple.  For each stratum $S$ let $\delta_S$ be a frontier limit function.
Then we set
\begin{eqnarray*}
	V_S &=& S - \bigcup_{T\prec S} \{ x\in U_T \mid \rho_T(x) < \delta_T\pi_T(x)\}\\
	V_{TS} &=& \{ x\in V_S\cap U_T \mid \rho_T(x) = \delta_T\pi_T(x)\}\\
	p_{TS} &=& \pi_T|_{V_{TS}}
\end{eqnarray*} 

\begin{lemma}\label{resolution}
	Suppose $\{ (U_S,\pi_S,\rho_S,\gamma_S)  \}$ is enhanced Thom data for a Thom stratified set $X$.
	Suppose for each stratum $S$ of $X$ we choose a smooth $\delta_S\colon S\to (0,\infty)$ so that $\delta_S<\gamma_S$
	everywhere.  Then the above $\{V_S,V_{TS},p_{TS}\}$ is an untico resolution tower with realization $|X|$.
	Moreover, this tower has some special properties:
	\begin{enumerate}
		\item Each $V_{TS}$ is itself a manifold with convex corners on the boundary (i.e., it has no interior creases).
		\item If $P\prec T\prec S$ then $p_{TS}^{-1}(V_{PT})= V_{PS}\cap V_{TS}$.
		\item Each $p_{TS}$ is a strong submersion.
		\item Up to isomorphism, the tower is independent of the choice of $\delta_S$.
	\end{enumerate}
\end{lemma}

\begin{proof}
	For convenience, we may as well reparameterize using Lemma \ref{reparam} so that $\delta_S = 1/2$ and $\gamma_S = 1$
	are both constant.
	
	Take any $x\in V_S$ and let $T_1,\ldots,T_k$ be all the strata with $\rho_{T_i}(x) = 1/2$.
	Let $\ell$ be the depth of $x$ in $S$.
	Since $\rho_{T_1}\times\cdots\times \rho_{T_k}$ is a strong submersion near $x$ we can choose local
	coordinates $y$ near $x$ so that $\rho_{T_i}(y) = 1/2 + y_i$ and $S$ is given by $y_i\ge 0$,
	$i=k+1,\ldots,k+\ell$.
	Then near $x$, $V_S$ is given by the inequalities $y_i\ge 0$, $i=1,\ldots,k+\ell$ and so $V_S$
	has convex corners and the depth of $x$ in $V_S$ is $k+\ell$.
	If $x\in V_{TS}$ then after reordering, $T=T_1$  and $V_{TS}$ is locally given by the equations
	$y_1=0$, $y_i\ge 0$ for $i=2,\ldots k+\ell$.
	Thus $V_{TS}$ is a manifold with convex corners on the boundary.
	If $x$ has depth 0 in $V_{TS}$ then $\ell=0$ and $k=1$, thus the depth 0 points of $V_{TS}$
	are a union of connected components of the depth 1 points of $V_S$.
	So $V_{TS}$ is a union of closures of components of the depth 1 points of $V_S$.
	If $x$ has depth 1 in $V_S$ then either $k=1$ and $\ell=0$ in which case $x\in V_{T_1S}$ and $x$ is in no other $V_{TS}$,
	or $k=0$ and $\ell=1$ in which case $x$ is in no $V_{TS}$.

	If $x\in V_{TS}\cap V_{PS}$ for $P\ne T$ then after reordering, $T=T_1$ and $P=T_2$.
	Since $x\in U_P\cap U_T$ we know either $P\prec T$ or $T\prec P$.
	Suppose $P\prec T$. For any $Q\prec T$ we have 
	$$\rho_Q p_{TS}(x ) = \rho_Q\pi_T(x) = \rho_Q(x) \ge 1/2$$
	and thus $p_{TS}(x)\in V_T$.
	Also $\rho_P p_{TS}(x) = \rho_P(x) = 1/2$ so $p_{TS}(x)\in V_{PT}$.
	We also have $p_{PT}p_{TS}(x)= \pi_P\pi_T(x) = \pi_P(x) = p_{PS}(x)$.
	
	Now suppose $x\in V_{TS}$ and $p_{TS}(x)\in V_{PT}$.
	Then $1/2  = \rho_P p_{TS}(x) = \rho_P\pi_T(x) = \rho_P(x)$ so $x\in V_{PS}$.
	
	Finally we must show $p_{TS}$ is a strong submersion.
	Take any $x\in V_{TS}$ and let $T_1,\ldots,T_k$ be all the strata with $\rho_{T_i}(x) = 1/2$.
	Order so that $T_1\prec T_2\prec\cdots\prec  T_k$.
	We know $T=T_n$ for some $n$.
	Since $\rho_{T_1}\times \cdots\times \rho_{T_{n-1}}$ submerses $T$
	we can choose cooordinates $z$ near $\pi_T(x)$ so that 
	in these coordinates,   $T$ is given by the inequalities
	$z_i\ge 0$ for $i = 1,\ldots,\ell$ and  $\rho_{T_i}(z) = 1/2 + z_{\ell+i}$ for $i=1,\ldots,n-1$  
	where $\ell$ is the depth of $\pi_T(x)$ in $T$
	(so $\ell$ is also the depth of $x$ in $S$).
	Since $\pi_{T_n}\times \rho_{T_n}\times \rho_{T_{n+1}}\times \cdots\rho_{T_k} $ is a strong 
	submersion we know by Lemma \ref{submersion}  we can choose coordinates $y$ near $x$ so that in these coordinates
	$S$ is given by $y_i\ge 0$ for $i = 1,\ldots,\ell$, $\pi_T$ is given by projection to the $z$ coordinates
	(i.e., if $m=\dim T$ the $z_i$ coordinate of $\pi_T(y)$ is $y_i$ for $i\le m$), and $\rho_{T_i}(y) = 1/2 + y_{m+i-n+1}$
	for $i= n,\ldots,k$.  Since $\rho_{T_i}=\rho_{T_i}\pi_T$ for $i<n$ we also know that $\rho_{T_i}(y) = 1/2 + y_{\ell+i}$
	for all $i<n$.
	We now see that $p_{TS}$ is a strong submersion.
	In particular, $V_T$ is given by the inequalities $z_i\ge 0$ for $i = 1,\ldots, \ell+n-1, m+1,m+2,\ldots,m+k-n+1$,
	$V_{TS}$ is given by the inequalities $y_{m+1}=0$, $y_i\ge 0$ for $i = 1,\ldots, \ell+n-1, m+2, m+3,\ldots, m + k -n +1$,
	and $p_{TS}$ is given by projection.
	
	So we have shown that $\{ V_S, V_{TS}, p_{TS}\}$ is an untico resolution tower
	with some very special properties.
	
	Let us now show that its realization is $|X|$.
	By Lemma \ref{proper_r}
	there is a smooth proper map $r\colon |X|\to [0,\infty)$ so that for all strata $S$
	and $x\in \rho_S^{-1}([0,2/3])$,
	$r\pi_S(x) = r(x)$.

	For any $s\in (0,1/2]$ let $V^s_S = S - \bigcup_{T\prec S} \{ x\in U_T \mid \rho_T(x) < s\}$ and
	$V_{TS}^s = \{ x\in V_S^s\cap U_T \mid \rho_T(x) = s\}$.
	Thus $V_S = V_S^{1/2}$ and $V_{TS} = V_{TS}^{1/2}$.
	We will define a vector field $v$ on $|X|$ whose flow $\phi_t$ is continuous,
	so that $\phi_t(V_S^s) = V^{s-t}_T$ and $\phi_t(V_{TS}^s)= V_{TS}^{s-t}$ for $t\in [0,s)$,
	so that $dr(v) = 0$, 
	so that $\pi_S\phi_t = \phi_t\pi_S$ and so $\lim_{t\to 1/2}\phi_t(x)$ exists for all $x\in V_S$
	and the limits give a continuous map $\phi_{1/2}\colon V_S\to \Cl S$ so that 
	for $x\in V_{TS}$, $\phi_{1/2}(x) = \phi_{1/2}\pi_T(x)$.
	Moreover $\phi_{1/2}$ restricts to a diffeomorphism of $V_S - \bigcup_{T\prec S}V_{TS}$ to $S$.
	Consequently, $\phi_{1/2}$ induces an isomorphism of the realization of the resolution tower to $|X|$.
	
	To construct $v$ it suffices to construct for each stratum $S$ a smooth vector field $v_S$ on $S$
	so that:
	\begin{itemize}
		\item  $v_S$ is tangent to $\partial_k S$ for each $k$,
		\item  $dr(v_S)=0$,
		\item  $d\rho_T(v_S(x)) = -1$ and  $d\pi_T(v_S(x))= v_T(\pi_T(x))$     if $x\in \rho_T^{-1}((0,1/2])$,
		\item  $d\rho_T(v_S(x))\ge -1$ if $x\in \rho_T^{-1}((1/2,1))$.
	\end{itemize}
	Furthermore it suffices to construct such a $v_S$ locally and piece together with a partition of unity.
	
	So suppose we have constructed $v_T$ for all $T\prec S$.
	Take any $y\in S$ and let us construct $v_S$ locally around $y$.
	Let $T_1\prec T_2\prec \cdots\prec T_k$ be all the strata so that $y\in U_{T_i}$.
	Let $j$ be the highest index so that $\rho_{T_j}(y)\le 1/2$.
	If there is no such $j$ we may locally set $v_S=0$.
	Since $\pi_{T_j}\times \rho_{T_j}\times \cdots\times \rho_{T_k}$ is a strong submersion we may choose 
	$v_S$ locally so that $d\pi_{T_j}(v_S) = v_{T_j}$, $d\rho_{T_j}(v_S) = -1$, and $d\rho_{T_i}(v_S)=0$
	for $i>j$.
	Note $dr(v_S) = drd\pi_{T_j}(v_S) = dr(v_{T_j}) = 0$.
	Also for $i<j$ we have $d\rho_{T_i}(v_S)=d\rho_{T_i}d\pi_{T_j}(v_S)= d\rho_{T_i}(v_{T_j}) \ge -1$
	and is $=-1$ if $\rho_{T_i}(x)\le 1/2$.
	Also if  $\rho_{T_i}(x)\le 1/2$ then 
	$d\pi_{T_i}(v_S)=d\pi_{T_i}d\pi_{T_j}(v_S)= d\pi_{T_i}(v_{T_j}) = v_{T_i}$.
	
	We must now show the tower is independent of the choice of $\delta_S$ so suppose we made other choices $\delta'_S$.
	It suffices to prove the case where $\delta'_S<\delta_S$ everywhere since we can always choose a third $\delta''_S$
	less than both.  We may as well reparameterize so that $\delta'_S = 1/4$, $\delta_S = 1/2$, and $\gamma_S=1$ everywhere.
	Then note that the untico resolution tower $\{V'_S,V'_{TS},p'_{TS}\}$ we get using $\delta'_S$ is exactly
	$V'_S = V^{1/4}_S$, $V'_{TS} = V^{1/4}_{TS}$ and $p'_{TS}=\pi_S|_{V'_{TS}}$.
	Then the map $\phi_{1/4}\colon V_S\to V'_S$ gives an isomorphism of the untico resolution towers.
	Note 
	$\phi_{1/4}p_{TS}(x)=\phi_{1/4}\pi_S(x) = \pi_S\phi_{1/4}(x) = p'_{TS}\phi_{1/4}(x)$ for all $x\in V_{TS}$.
\end{proof}

An earlier version of the paper required the following result.
It is no longer required, but we include it anyway.
It could be easily generalized to have $q$ be any proper strong submersion to a manifold $Q$ with convex corners on the
boundary, and then $\{V_S,V_{TS},p_{TS}\}$ would fiber over $Q$, but let's keep it simple.

\begin{lemma}\label{concordance_weak}
	Suppose $\{ (U_S,\pi_S,\rho_S,\gamma_S)  \}$ is enhanced Thom data for a Thom stratified set $X$
	  and $q\colon |X|\to [0,1]$ is a proper map which restricts to a strong submersion
	on each stratum of $X$ so that $q\pi_S = q|_{U_S}$ for all strata $S$.
	 Let $\{V_S,V_{TS},p_{TS}\}$ be the resulting untico resolution tower of $X$.
	Then there is an untico resolution tower $\{V'_S,V'_{TS},p'_{TS}\}$ so that $\{V_S,V_{TS},p_{TS}\}$
	is isomorphic to  $\{V'_S,V'_{TS},p'_{TS}\}\times [0,1]$.
	In particular, for each stratum $S$ there are diffeomorphisms $h_S\colon V'_S\times [0,1]\to V_S$ so that
	$qh_S(x,t) = t$, $h_S(V'_{TS}\times [0,1])= V_{TS}$, and $p_{TS}h_S(x,t) = h_T(p'_{TS}(x),t)$.
	Moreover $\{V'_S,V'_{TS},p'_{TS}\}$ is an untico resolution tower for any of the Thom stratified sets $X\cap q^{-1}(t)$
	with Thom data obtained by restricting $\pi_S$ and $\rho_S$ to $U_S\cap q^{-1}(t)$.
\end{lemma}

\begin{proof}
	As usual, we may as well suppose that $\gamma_{S}=1$.
	The next step is to find a controlled vector field $v$ on $|X|$ so that $dq(v) = 1$.
	In particular, for each stratum $S$ of $X$ we want a smooth vector  field $v_S$ on $S$
	so that 
	\begin{itemize}
		\item $dq(v_S)=1$.
		\item If $T\prec S$, $x\in S\cap U_T$ and $\rho_{T}(x)\le 1/2$ then 
		$d\pi_{T}v_S(x)= v_T\pi_{T}(x)$ and $d\rho_{T}v_S(x)=0$.
	\end{itemize}
	It suffices to find $v_S$ locally and piece together with a partition of unity
	and it is easy to find $v_S$ locally, c.f., the proof of Lemma \ref{resolution}.
	Now integrate the $v_S$ to find a flow $\phi_t$ on $|X|$.
	Let $\{V_S,V_{TS},p_{TS}\}$ be the untico resolution tower we obtain using $\delta_{S}=1/2$.
	Note that $\phi_t$ leaves each $V_S$ and $V_{TS}$ invariant and 
	we may set $V'_S=V_S\cap q^{-1}(0)$, $V'_{TS}=V_{TS}\cap q^{-1}(0)$, $p'_{TS}=p_{TS}|_{V'_{TS}}$,
	and $h_S(x,t) = \phi_t(x)$.
\end{proof}

\section{Whitney stratified subsets}

A stratified set $X$ is a Whitney stratified subset of a smooth manifold $M$ with convex corners on the boundary
if $|X|\subset M$ and all strata of $X$ are smooth submanifolds of $M$ so that:
\begin{itemize}
	\item Each stratum has boundary compatible with $M$.
	\item If $T\ne S$ are strata of $X$ and $x\in T\cap \Cl S$  then $(T, S)$ satisfies the Whitney conditions A and B at $x$
	which we define below.
\end{itemize}

Let us recall the Whitney conditions.
Suppose $S$ and $T$ are disjoint smooth submanifolds
of some smooth manifold $M$ with convex corners on the boundary
and both $S$ and $T$ have boundary compatible with $M$.
We say that $(T, S)$ satisfies the Whitney
conditions at a point $x\in T$ if
\begin{itemize}
	\item Condition A: Whenever $x_i\in S$ is a sequence converging to $x$
	so that the tangent spaces $T_{x_i}S$ of $S$ converge to
	some subspace $L$ of $T_xM$, then $T_xT\subset L$,
	the tangent space of $T$ at $x$ is contained in $L$.
	\item Condition B: Whenever $x_i\in S$ is a sequence converging to $x$
	and $x_i'\in T$ is a sequence converging to $x$ so that
	the tangent spaces $T_{x_i}S$ of $S$ converge to
	some subspace $L$ of $T_xM$ and the secant
	lines $x_ix_i'$
	converge to some line $\ell$ in $T_xM$, then $\ell \subset L$.
\end{itemize}
In the above, the secant lines are taken in some local
coordinates, it doesn't matter which.
Also, condition A is superfluous, it is implied by B
(just let $x_i$ approach $x$ much faster than $x_i'$
and $\ell$ can be any vector in $T_xT$).
However,  it is sometimes useful to prove
condition A anyway as this will help when proving B.

We say that $(T, S)$ satisfies the Whitney conditions
if it satisfies them at every point $x\in T$.

Whitney stratified sets have Thom data and hence are Thom stratified sets.
The reader should be able to adapt any standard proof of this to our
convex corner generalization of Whitney stratified set, but we prove it here anyway.
In particular, in \cite{Tico} we proved a stronger result- that if you choose
any $\rho_S$ you wish, subject to a mild condition that each $\rho_S$ look 
locally like a squared distance function, then you may construct $\pi_S$
so that $\{(U_S,\pi_S,\rho_S)\}$ is Thom data.

Statements and proofs given in the Whitney appendix of \cite{Tico} remain true 
for the convex corner case as long as we make the following modifications:
\begin{itemize}
	\item Manifolds are allowed to have convex corners on the boundary but submanifolds 
	are required to have compatible boundary.
	\item Replace submersion with strong submersion.  Keep in mind that Lemma \ref{depth_submersion} above shows that a depth preserving submersion is a strong submersion.
	\item Specify that maps preserve depth where appropriate.
	\item Replace
	parameter domains $V\subset \R^m$ with $V\subset [0,\infty)^\ell\times \R^{m-\ell}$.
\end{itemize}
We reproduce the suitably modified statements of the results below.
The proofs will be the same as in \cite{Tico} once the above modifications are made.
We start by recalling some definitions from \cite{Tico}, suitably modified.

Suppose $M$ is a smooth manifold with convex corners on the boundary,  $N$ is a smooth submanifold of $M$
whose boundary is compatible with $M$, and
$\rho\colon U\to [0, \infty)$ is a smooth function from an
open \nbhd\ $U$ of $N$ in $M$.
We say that $\rho$ is {\em distancelike}
around $N$ if:
\begin{itemize}
	\item At every point of $N$
	the Hessian of $\rho$ has rank equal to the codimension of $N$.
	\item $N\subset \rho^{-1}(0)$.
\end{itemize}
Lemma 6 of \cite{Tico} gives a local description of a distancelike function.
Suppose $z\in N$ and $h\colon(V,V\cap N,z)\to [0,\infty)^k\times(\R^{n-k}\times\R^{m-n},\R^{n-k}\times 0,0)$
is a coordinate chart around $z$.  Then there is a  symmetric matrix valued function $L(x,y)$
so that $\rho h^{-1}(x,y) = y^TL(x,y)y$ and every $L(x,0)$ is positive definite.
After shrinking $V$ and changing the $y$ coordinate you then get $\rho h^{-1}(x,y) = |y|^2$.

Lemma 14 of \cite{Tico} is modified to:

\begin{lemma}\label{Whitney_submersion}
	Suppose $M$ is a smooth manifold with convex corners on the boundary,
	$T$ and $S$ are submanifolds of $M$ with compatible boundary,
	$(T, S)$ satisfies the Whitney conditions,
	$U$ is a \nbhd\ of $T$ in $M$,
	$\rho\colon U\to [0,\infty)$ is distancelike around $T$,
	and $\pi\colon U\to T$ is a smooth retraction which preserves depth.
	Then there is a \nbhd\ $U'$ of $T$ in $U$ so that
	$\rho\times \pi|\colon U'\cap S\to (0,\infty)\times T$
	is a strong submersion.
\end{lemma}

Lemmas \ref{frontier} and \ref{frontier_proper} above take the place of Lemma 16 of \cite{Tico}.
Much of Lemma 17 of \cite{Tico} is contained in Lemma \ref{goodthom} but the rest of it is modified to:

\begin{lemma}\label{piconstruct}
	Suppose $M$ and $N$ are smooth manifolds with convex corners on the boundary,
	$X$ is a Whitney stratified subset of $M$,
	$q\colon M\to N$ is a strong submersion which preserves depth, and
	$q$ restricts to a strong submersion on each stratum of $X$.
	Suppose that for each stratum $S$ of $X$ we choose a
	distancelike
	function $\rho_S\colon U''_S\to [0,\infty)$
	around  $S$. 
	Moreover, suppose that there is a closed set $Y\subset |X|$ and for each stratum
	$S$ we have a (possibly empty) \nbhd\ $U'_S$ of $Y\cap S$ in $U''_S$ and 
	a smooth retraction $\pi'_S\colon U'_S\to S\cap U'_S$ 
	so that $q\pi'_S=q$ and if $T\prec S$ then $\pi'_T\pi'_S = \pi'_T$ and
	$\rho_T\pi'_S = \rho_T$
	on $\pi'_S{}^{-1}(U'_T)\cap U'_T$.
	Then there are \nbhd s $U_S$ of $S$ in $U''_S$ and
	smooth retractions $\pi_S\colon U_S\to S$  so that $q\pi_S=q|_{U_S}$ for all strata $S$ and
	$\{(U_S,\pi_S,\rho_S|_{U_S})\}$ is Thom data for $X$.
\end{lemma}

The proof of Lemma \ref{piconstruct} uses the technical Lemma 13 of \cite{Tico} and the notion of locally linear.
Suppose $X$, $Y$, and $Z$ are smooth manifolds with convex corners on the boundary, $Y\subset X$
is a submanifold with boundary compatible with $X$, and  $f\colon U\to Z$ is a smooth map where $U\subset X$ is open.
We say $f$ is locally linear with respect to $Y$ if for every $x\in U\cap Y$ we may choose local
coordinates $g\colon (V, x)\to [0,\infty)^k\times(\R^m,0)$ and 
$h\colon (W,f(x))\to [0,\infty)^\ell\times(\R^n,0)$
so that 
$V\cap Y = g^{-1}([0,\infty)^k\times E)$ for some linear subspace $E$ of $\R^m$,
$hfg^{-1}(y,z)=L(y,z)$ for some linear transformation $L\colon \R^k\times \R^m\to \R^\ell\times \R^n$,
and  $f^{-1}f(x)$ is transverse to $Y$, i.e., $L(\R^k\times E) = L(\R^k\times\R^m)$.
Then Lemma 13 of \cite{Tico} becomes:

\begin{lemma}\label{make_retract_rel}
	Suppose $M$ is a smooth manifold with convex corners on the boundary,
	$N\subset M$ is a smooth submanifold with boundary compatible with $M$,
	$\{U_i\}_{i\in A}$ is a locally finite  collection of open subsets of $M$,
	$q_i\colon U_i\to Z_i$ are smooth maps
	to manifolds $Z_i$, and $C_i\subset M$ are closed subsets so that $C_i\subset U_i$. 
	For each nonempty subset
	$D\subset A$ with  $\bigcap_{i\in D} U_i$ nonempty we suppose that the map
	$\Pi_{i\in D}q_i\colon \bigcap_{i\in D} U_i\to \Pi_{i\in D}Z_i$
	is locally linear with respect to $N$.
	Suppose $Y\subset N$ and $K\subset N$ are closed subsets of $N$,  $U$ is a \nbhd\
	of $Y$ in $M$, and $\sigma\colon U\to U\cap N$ is a smooth 
	retraction  so that $\sigma$ preserves depth and
	$q_i\sigma(x)=q_i(x)$ for all $i$ and
	all $x$ in $U_i\cap \sigma^{-1}(U_i\cap N)$.
	Then there is a \nbhd\ $V$ of $K\cup Y$ in $M$ and a smooth
	retraction $\pi\colon V\to V\cap N$ so that:
	\begin{enumerate}
		\item $\pi(x)=\sigma(x)$ for all $x$ in some \nbhd\ of $Y$.
		\item $q_i\pi(x)=q_i(x)$ for all $i$ and
		all $x$ in some \nbhd\ of $C_i\cap \pi^{-1}(C_i\cap N)$.
		\item $\pi$ preserves depth.
	\end{enumerate}
\end{lemma}

The generalization of the following Lemma to Thom stratified sets is of course false.
For example, take nondiffeomorphic compact manifolds $M$ and $M'$ with diffeomorphic interiors.
Then $M/\partial M$ and $M'/\partial M'$ are isomorphic stratified sets but if we take $\rho$ and $\rho'$
arising from collars of $M$ and $M'$ respectively the untico resolution towers differ. 
One is $V_1=M, V_0 = *, V_{01}=\partial M$ and the other is $V'_1=M', V_0 = *, V'_{01}=\partial M'$.

\begin{lemma}\label{untico_isomorphic}
	Suppose $X$ is a Whitney stratified subset of a smooth compact manifold $M$ with convex corners on the boundary
	and $\{(U_S,\pi_S,\rho_S)\}$ and  $\{(U'_S,\pi'_S,\rho'_S)\}$
	are two sets of Thom data for $X$ so that all $\rho_S$ and $\rho'_S$ are distancelike.
	Then the two untico resolution towers $\{V_S,V_{TS},p_{TS}\}$ and  $\{V'_S,V'_{TS},p'_{TS}\}$
	we obtain are isomorphic.
\end{lemma}

\begin{proof}
	Consider the Whitney stratified set $X\times [0,1]$ in $M\times [0,1]$.
	Choose a smooth $\alpha\colon [0,1]\to [0,1]$ so that $\alpha(t) = 1$ for all $t>3/4$ and $\alpha(t) = 0$ for all $t<1/4$.
	For each stratum $S\times[0,1]$ of $X\times[0,1]$ let $\rho''_{S\times[0,1]}$ be the distancelike function 
	$$\rho''_{S\times[0,1]}(x,t) =  \alpha(t)\rho_S(x) + (1-\alpha(t))\rho'_S(x)$$ on $(U_S\cap U'_S)\times[0,1]$.
	By Lemmas \ref{piconstruct} and \ref{goodthom} 
	there are smooth retractions $\pi''_{S\times[0,1]}\colon U''_{S\times [0,1]} \to S\times[0,1]$
	and smooth $\gamma''_{S\times [0,1]}\colon S\times [0,1] \to (0,\infty)$
	so that:
	\begin{itemize}
		\item $\{ (U''_{S\times [0,1]},\pi''_{S\times[0,1]}, \rho''_{S\times [0,1]},\gamma''_{S\times [0,1]})  \}$ is enhanced Thom data for $X\times [0,1]$.
		\item $\pi''_{S\times[0,1]}(x,1) = (\pi_S(x), 1)$ for all $(x,1)\in U''_{S\times [0,1]}$.
		\item $\pi''_{S\times[0,1]}(x,0) = (\pi'_S(x), 0)$ for all $(x,0)\in U''_{S\times [0,1]}$.
		\item $q\pi''_{S\times[0,1]}(x,t) =  t$ where $q\colon M\times[0,1]\to [0,1]$ is projection.
	\end{itemize}
	Let $\{V''_S,V''_{TS},p''_{TS}\}$ be the untico resolution tower we obtain using the enhanced Thom data
	 $\{ (U''_{S\times [0,1]},\pi''_{S\times[0,1]}, \rho''_{S\times [0,1]},\gamma''_{S\times [0,1]})  \}$
	and some
	$\delta''_{S\times[0,1]}$.
	Note that 
	$V''_S\cap M\times 1 = V_S\times 1$, $V''_{TS}\cap M\times 1 = V_{TS}\times 1$,
	$V''_S\cap M\times 0 = V'_S\times 0$, and $V''_{TS}\cap M\times 0 = V'_{TS}\times 0$
	for an appropriate choice of the $\delta_S$ and $\delta'_S$.
	So by  Lemma \ref{concordance_weak} $\{V_S,V_{TS},p_{TS}\}$ and  $\{V'_S,V'_{TS},p'_{TS}\}$
	are isomorphic.
\end{proof}

\section{Generic Morse cells satisfy the Whitney conditions}

We could fancify the following to allow convex corners on the boundary of $M$ and have
$f$ be locally constant on some components of each $\partial_k M$ and submerse other components,
$k>0$.  I leave it to the reader to correctly state all that.  The proof will be the same.

Recall that a gradientlike vector field for a Morse function $f$ on $M$ is a vector field $v$ on $M$
so that $v(f)\ge 0$ everywhere, $v(f)=0$ only at critical points of $f$, 
and near each critical point there are local coordinates
$(x,y)$ so that in these coordinates $f(x,y)=c+|y|^2-|x|^2$
and $v(x,y)=(-x,y)$.

Let $\phi_t$ be the flow on $M$ obtained by integrating $v$.
The stable manifold of a critical point $p$ is
$S_p=\{q\in M \mid \lim_{t\to\infty} \phi_t(q) = p\}$ and the
unstable manifold of $p$ is
$U_p = \{q\in M \mid \lim_{t\to-\infty} \phi_t(q) = p\}$.

The reason for writing this paper was to prove the following result.

\begin{lemma}\label{CW}
	Let $f\colon M\to \R$ be a Morse function on a compact smooth manifold $M$ possibly with boundary
	(but without corners on the boundary).
	Suppose $f$ is locally constant on $\partial M$ and has no critical points on $\partial M$.
	Suppose that $v$ is a gradientlike vector field on $M$ so that each of its stable manifolds 
	is transverse to each of its unstable manifolds.
	Let $M'\subset M$ be the union of the stable manifolds.
	Then:
	\begin{enumerate}
		\item The stratification of $M'$ by
		stable manifolds is a Whitney stratification.
		\item For each regular value $c$ of $f$,
		the stratification of $M'\cap f^{-1}(c)$ by the intersections with
		stable manifolds is a Whitney stratification.
		\item If $p$ and $q$ are critical points of $f$ and $S_p$ intersects $\Cl S_q$, then $S_p\subset \Cl S_q$
		and $f(p)<f(q)$.
		\item $M'$ is a closed subset of $M$.
		\item If $\partial M$ is empty,
		there is a CW complex structure on $M$ so that the
		stable manifolds are the (interiors of) the cells.
		\item If all points of $\partial M$ are local minima of $f$,
		there is a relative CW complex structure on $(M,\partial M)$ so that the
		stable manifolds are the (interiors of) the  cells.
		\item In general, $(M',M'\cap \partial M)$ has a relative CW complex structure  so that the
		stable manifolds are the  (interiors of) the cells.
	\end{enumerate}
\end{lemma}

\begin{proof}
To see that $M'$ is closed, let $\phi_t$ be the flow generated by $v$ and take any $x\in M-M'$.
If $\phi_t(x)$ is defined for all $t\ge 0$ then by compactness of $M$ there must be a sequence
$t_i\to \infty$ so that $\phi_{t_i}(x)\to$ some $ p$.  We must have $v(p)=0$, hence $p$ is a critical point
and $x\in S_p\subset M'$, a contradiction.  So there is a $t_0>0$ so that $\phi_{t_0}(x)\in \partial M$.
But then by continuity of $\phi$, for all $y$ near $x$ there is a $t_y$ so that $\phi_{t_y}(y)\in \partial M$
and thus all $y$ near $x$ are not in $M'$. 

Note the second item follows immediately from the first since if $S,T$ satisfies the Whitney conditions
and $N$ is transverse to $S$ and $T$, then $N\cap S, N\cap T$ satisfies the Whitney conditions.

The third item also follows from the first.  $S_p\cap \Cl S_q$ is closed in $S_p$, but it is also open in
$S_p$ because of local triviality of Whitney stratifications.
However, to keep things elementary we will provide another proof below.

Let $p$ be a critical point of $f$.
We assume by induction that $S_r,S_q$ satisfies the Whitney
conditions for all stable manifolds $S_r$ and $S_q$ with $f(r)>f(p)$.
We also assume by induction that if $S_r$ intersects $\Cl S_q$ and $f(r)>f(p)$ then $S_r\subset \Cl S_q$ and $f(q)>f(r)$.

Suppose $p$ has index $k$.
Near $p$ we have nice local coordinates,
$h\colon (V_p,p)\to \R^k\times \R^{n-k}$
  with:
\begin{itemize}
	\item $fh^{-1}(x,y)=f(p)+|y|^2-|x|^2$.
	\item $dh(v)(x,y)=(-x,y)$.
	\item $h\phi_th^{-1}(x,y) = (e^{-t}x,e^ty)$.
	\item The stable manifold of $p$ near $p$ is $S_p\cap V_p = h^{-1}(\{(x,0)\in V_p\})$.
	\item The unstable manifold of $p$ near $p$ is $U_p\cap V_p=h^{-1}(\{(0,y)\in V_p\})$.
\end{itemize}
For convenience we may suppose  $h(V_p)$ is
the set of $(x,y)$ with $|x|<2\epsilon$ and $|y|<2\epsilon$.
Note we don't have much control over what another stable manifold
 $S_q$ looks like in these coordinates, other than it being 
invariant under $\phi_t$.

\begin{assertion}\label{ax_fron}
	Suppose $q\ne p$ is another critical point and $U_p\cap S_q$ is nonempty.
	Then $S_p\subset \Cl S_q$ and $f(q)>f(p)$.
\end{assertion}

\begin{proof}
	Suppose $z\in U_p\cap S_q$.
Since stable and unstable manifolds are invariant under $\phi_t$,
by taking $t$ to be  large and negative 
we get  $h\phi_t(z) =(0,y_0)\in h( U_p\cap S_q)$.
Since $S_q$ is transverse to $U_p$
we know that projection to the $\R^k$ coordinate submerses $h(S_q)$ near $h(U_p)$,
so there is a smooth embedding $\kappa \colon U\to S_q$
where $U$ is a \nbhd\ of 0 in $\R^k$ and $h\kappa(x) = (x,\kappa'(x))$ for some $\kappa'$
with $\kappa'(0)=y_0$.
Pick any $x_0\in \R^k$ near 0. Then $h\phi_{-i}\kappa(e^{-i}x_0) = (x_0,e^{-i}\kappa'(e^{-i}x_0))$ $i=1,2,\ldots$ is a sequence of points in
$h(S_q)$ which approaches $(x_0,0)$ so $(x_0,0)\in h(S_p\cap \Cl S_q)$.
But any point in $S_p$ is $\phi_th^{-1}((x,0))$ for some $t$ and some $x\in \R^k$ near 0,
so $S_p\subset \Cl S_q$.
Since $q\ne p$ we know $y_0\ne 0$ so
 $f(q)\ge fh^{-1}(0,y_0) = f(p)+|y_0|^2 > f(p)$.	
\end{proof}

\begin{assertion}\label{trick}
	Suppose $x_i\in S_q$ is a sequence and $x_i\to x_0\in S_p$, $q\ne p$.
	Then after perhaps taking a subsequence there is a sequence $t_i\to\infty$ so that:
	\begin{enumerate}
		\item $h\phi_{t_i}(x_i)=(x_i',y_i')$ and $h\phi_{t_0}(x_i)=(x_i'',y_i'')$,
		 $i=0,1,\ldots$.
		\item $(x_i',y_i')\to (0,y_0)$ for some $y_0\ne 0$.
		\item $(0,y_0)\in h(S_r)$ for some critical point $r$
		with $f(r)>f(p)$.
		\item If $L_i$ is the tangent space to $h(S_q)$ at $(x_i',y_i')$ then $L_i\to L$ with $L$
		transverse to $0\times \R^{n-k}$.
	\end{enumerate}
\end{assertion}

\begin{proof}
	Choose $t_0$ large enough that $\phi_{t_0}(x_0)\in V_p$ and let  $h\phi_{t_0}(x_0)=(x_0',0)$.
	After taking a subsequence we may suppose each $\phi_{t_0}(x_i)\in V_p$.
	Let $h\phi_{t_0}(x_i) = (x_i'',y_i'')$.
	Since $\phi_{t_0}(x_i)\to \phi_{t_0}(x_0)$ we know $x_i''\to x_0'$ and $y''_i\to 0$.
	Let $t_i = t_0 + \ln \epsilon - \ln |y_i''|$.
	Then $h\phi_{t_i}(x_i) = (|y_i''|x_i''/\epsilon, \epsilon y_i''/|y_i''|)=(x_i',y_i')$.
	After taking a subsequence we may assume $(x_i',y_i')\to (0,y_0)$ and $L_i\to L$ for some 
	$y_0$ and $L$.
	Since $M'$ is closed we know $(0,y_0)\in h(S_r)$ for some $r$.
	We have $f(r)\ge fh^{-1}(0,y_0) = f(p) + |y_0|^2 = f(p) + \epsilon^2 > f(p)$.
	
	If $r=q$ then $L$ is the tangent space
	to $h(S_q)$ at $(0,y_0)$.
	Since $S_q$ is transverse to $U_p$ we know that $L$ is transverse to 
	$0\times \R^{n-k}$.
	If $r\ne q$ we know by our induction assumption that $S_r,S_q$ satisfies
	the Whitney conditions
	 so $L$ contains the tangent space to 
	$h(S_r)$ at $(0,y_0)$ which by assumption is transverse to $h(U_p)$.
	Hence $L$ is transverse to $0\times \R^{n-k}$.
\end{proof}

Pick any critical point $q\ne p$ so that $ S_p\cap \Cl S_q$ is nonempty.
We need to show that $S_p\subset \Cl S_q$.
By Assertion \ref{ax_fron} we reduce to the case where 
 $U_p\cap S_q$ is empty.
In fact this cannot happen, but let's pretend it does.
By Assertion \ref{trick} there is a critical point $r$
so that $U_p\cap S_r$ is nonempty and $S_r\cap \Cl S_q$ is nonempty and $f(r)>f(p)$.
By our induction assumption we know $S_r\subset \Cl S_q$ and $f(q) > f(r)$.
By Assertion \ref{ax_fron}, we know $S_p\subset \Cl S_r$.
Hence $S_p\subset \Cl S_r \subset \Cl \Cl S_q = \Cl S_q$ and $f(q)>f(r)>f(p)$.
So we have proven the third item.

Now let us turn our attention to proving that $S_p,S_q$ satisfies Whitney conditions at a point $x_0\in S_p$.
Take sequences $x_i\in S_q$ and $z_i\in S_p$ so that $x_i\to x_0$, $z_i\to x_0$, 
the tangent spaces to $S_q$ at $x_i$ converge to some $L_0$ and the secant lines $|z_ix_i|$ 
converge to some $\ell$.
We need to show that $\ell \subset L_0$.
Pick $t_i$, $L_i$, $L$, etc.~satisfying the conclusions of Assertion \ref{trick}
 and let $(z_i',0) = h\phi_{t_0}(z_i)$.
Now  $$dhd\phi_{t_0}\ell = \lim (x_i''-z'_i,y_i'')/\sqrt{|x_i''-z'_i|^2+|y_i''|^2} = (w,ay_0)$$
for some $w\in \R^k$ and $a\in \R$.

Since $L$ is transverse to 
$0\times \R^{n-k}$, $L$ contains a subspace
of the form $\{(x,Ax)\}$ for some linear transformation 
$A\colon \R^k\to \R^{n-k}$.
Since $L_i\to L$,
for large enough $i$ we know
$L_i$
contains a subspace of the form
$\{(x,A_ix)\}$ for some linear transformation 
$A_i\colon \R^k\to \R^{n-k}$,
and $A_i\to A$.
(For example, $A$ could be obtained from $L$ 
by some algorithmic Gaussian elimination process
and the same algorithm applied to $L_i$ would produce $A_i$ converging to $A$.)

Note that in block form, 
the derivative $dh\phi_{t}h^{-1} = 
\begin{bmatrix}
e^{-t}I & 0\\
0 & e^{t}I
\end{bmatrix}
$.
Let $L_i'=dh\phi_{t_0-t_i}h^{-1}L_i$ denote the tangent space to $h(S_q)$ at $(x_i'',y_i'')$.
Note $e^{t_0-t_i}(w,A_iw)\in L_i$ so applying $dh\phi_{t_0-t_i}h^{-1}$ we get
$(w,e^{2t_0-2t_i}A_iw)\in L_i'$.
But $e^{t_0-t_i} = |y_i''|/ \epsilon\to 0$ so $(w,e^{2t_0-2t_i}A_iw)\to (w,0)$ and thus $(w,0)\in dh\phi_{t_0}h^{-1}L_0$.
Note $L_i$ must contain the tangent to the flow curve at $(x_i',y_i')$ which is $(-x_i',y_i')$.
Thus $L_i$ must contain $(-x_i',y_i')+(x_i',A_ix_i') = (0,y_i' + A_ix_i')$.
So $dh\phi_{t_0-t_i}h^{-1}(0,e^{t_i-t_0}(y_i' + A_ix_i')) = (0,y_i' + A_ix_i') \in L_i'$.
But $ (0,y_i' + A_ix_i') \to (0,y_0)$ and thus $(0,y_0)\in  dh\phi_{t_0}h^{-1}L_0$.
So $dh\phi_{t_0}h^{-1}\ell = (w,ay_0)\in dh\phi_{t_0}h^{-1}L_0$ and thus $\ell\in L_0$.

To finish it suffices to prove conclusion 7, since 5 and 6 are special cases of 7.
By Lemma \ref{resolution} it suffices to prove that if we take an untico resolution tower $\{V_S,V_{TS},p_{TS}\}$
of $M'$ then all the $V_S$ are homeomorphic to discs of various dimensions.
I will give two proofs, the first using results unknown at the time I first noticed this result and the second more elementary.
For the first method, note that the interior of $V_S$ is diffeomorphic to $S$ which is diffeomorphic to some $\R^k$.
Then the following assertion shows each $V_S$ is a disc.

\begin{assertion}
		Let $D$ be a compact topological manifold whose interior is homeomorphic to $\R^k$.
		Then $D$ is homeomorphic to the 
		$k$ dimensional disc.
\end{assertion}

\begin{proof}
	I won't bother explaining terms used here since the reader either knows this already
	or can look up this result elsewhere.
	The end of the interior of $D$ is collared by $\partial D\times [0,1)$ and by $S^{k-1}\times [0,1)$.
	Any two collarings of an end of a manifold differ by an invertible cobordism, so $\partial D$
	is invertibly cobordant to $S^{k-1}$, and hence is a homotopy sphere.
	By the Poincar\'e conjecture, now a theorem in all dimensions in Top, we know $\partial D$ is then 
	homeomorphic to $S^{k-1}$.
	By the Top Schoenflies theorem, $D$ is homeomorphic to the $k$ disc.	
\end{proof}

The second proof follows from the following assertion by following flow curves.

\begin{assertion}
	There are enhanced Thom data $\{(U_S,\pi_S,\rho_S,\gamma_S)\}$ for $M'$ and frontier limit functions $\delta_S$ so that the flow curves $\phi_t$ of $v$ are transverse to each $V_{TS}$ and point into $V_S$.
\end{assertion}

\begin{proof}
	Take any stratum $S$, the stable manifold $S_p$ of some critical point $p$ of index $k$.
	As above we take nice coordinates $h\colon (V_p,p)\to \R^k\times \R^{n-k}$ around $p$.
	After shrinking $V_p$ we may as well suppose there are 
	$\epsilon_S>0$ and $\epsilon_S'>0$ so that $h(V_p) = \{ (x,y)\mid \epsilon_S' > |x||y|, -\epsilon_S < |y|^2-|x|^2 < \epsilon_S   \}$.
	Let $V_p'$ be the set of $z\in M$ so that for some $t\ge 0$, $\phi_t(z)\in V_p$.
	Note that $V_p\subset V_p'$ and $V_p'$ is a \nbhd\ of $S$ in $M$.
	We may extend $h$ to $V_p'$ by setting $h(z) = (ux,y/u)$ where we choose $t\ge 0$ so $(x,y) = h\phi_t(z)$ and $u$ is the
	unique positive number so that $f(z) = f(p) + |y|^2/u^2 - |x|^2u^2$.
	This is independent of our choice of $t$.
	This extension $h$ has many of the nice properties of the original,
	$fh^{-1}(x,y) = f(p)+|y|^2-|x|^2$, 	$S = V_p'\cap S = h^{-1}(\R^k\times 0)$, and
	$h\phi_th^{-1}(x,y)= (u(t)x,\ y/u(t))$ for some smooth positive function $u$ (depending on $x$, $y$, and $S$).
	Note that $u'(t)<0$ since $0< df\phi_th^{-1}(x,y)/dt = -2u'(t)(|y|^2/u^3(t)+u(t)|x|^2)$.
	By Lemmas \ref{piconstruct} and \ref{goodthom} we may find enhanced Thom data $\{(U_S,\pi_S,\rho_S,\gamma_S)\}$ for $M'$
	so that $U_S\subset V'_p$ and $\rho_Sh^{-1}(x,y) = |y|^2$.
	Note that the flow curves of $v$ are transverse to the level sets $\rho_S^{-1}(\delta)$ and flow in the direction
	of increasing $\rho_S$.
	So we need only choose our frontier limit functions $\delta_S$ so that they are a constant on the compact set
	$V_S = S-\bigcup_{T\prec S} \{x\in S\cap U_T \mid \rho_T(x)< \delta_T\pi_T(x)   \}$.	
\end{proof}
\end{proof}

 \section{Uniqueness of Special Coordinates}
 
 For no other reason than mathematical amusement, we now turn
 to the question of how unique special coordinates are:
 given two special local coordinates around a critical point,
 what can you say about the local diffeomorphism $h$ which changes
 one to the other.  The equations are:
 \begin{eqnarray}
v\circ h &=& dh \circ v\\
f\circ h &=& f
\end{eqnarray}
Plugging 0 into equation (1) we have $v(h(0)) = dh(0) = 0$
so $h(0)=0$.  If $A$ is the linear part of the Taylor series
of $h$, then the quadratic part of equation (2) is
$fA=f$ from which we deduce that 
$A=\begin{pmatrix} B&0\cr0&C\end{pmatrix}$
where $B$ and $C$ are orthogonal.
Since this linear part itself preserves special local
coordinates and the special coordinate preserving (local) diffeomorphisms form a group, we may as well suppose that 
the linear part $A$ is the identity.

From equation (1), we know that $h$ preserves the flow of $v$,
so $\phi_th = h\phi_t$ which means
\begin{eqnarray}
e^th_i(x,y) &=& h_i(e^tx,e^{-t}y)\ \mathrm{if}\  i\le k\\
e^{-t}h_i(x,y) &=& h_i(e^tx,e^{-t}y)\ \mathrm{if}\  i> k
\end{eqnarray}
for $t$ near 0, or more precisely, as long as $\phi_t$ 
remains within our coordinate patch.
Combining this with equation (2) we see\footnote{
Multiply (2) by $e^{2t}$,
take $d/dt$ and Bob's your uncle.} that in fact
equation (2) splits up as two equations:
\begin{eqnarray}
|x|^2 &=& \Sigma_{i=1}^k h_i(x,y)^2\\
|y|^2 &=& \Sigma_{i=k+1}^n h_i(x,y)^2.
\end{eqnarray}

Note that if $y=0$ then equation (3) is valid for all
$t\le 0$ so $h(sx,0) = sh(x,0)$ for all $s\in [0,1]$.
From this we conclude\footnote{Write $h(x,0) = x + \Sigma_{i=1}^k
\Sigma_{j=1}^i x_ix_jg_{ij}(x)$ for some smooth $g_{ij}$,
then we get $s\Sigma_{i=1}^k\Sigma_{j=1}^i x_ix_jg_{ij}(sx)
= \Sigma_{i=1}^k\Sigma_{j=1}^i x_ix_jg_{ij}(x)$. Letting $s\to 0$
we conclude the right hand side is 0.}
 $h(x,0)=x$.  Likewise, $h(0,y)=y$.
 
 Consequently, if $k=0$ or $k=n$ or $n=2$ then $h$ is the identity,
 so any special local coordinate preserving diffeomorphism is linear, in fact orthogonal.
 But in all other cases there are nonlinear examples as we will show.

 Pick $\epsilon$ small enough that our coordinate charts
 contain the $2\epsilon$ ball.
 We have a parameterized family $f_{y,s}$ of diffeomorphisms
of the unit $k-1$ sphere and a parameterized family
$g_{x,s}$ of diffeomorphisms of the unit $n-k-1$ sphere,
for $y$ in the unit $n-k-1$ sphere and $x$ in the unit
$k-1$ sphere and $s\in (0,\epsilon]$ defined by:
$$h(\epsilon x, sy) = (\epsilon f_{y,s}(x), s g_{x,s}(y)).$$
Note that as $s\to 0$ then $h(\epsilon x, sy) \to \epsilon x$
so $f_{y,s}\to$ the identity.
Also, by equations (3) and (4) with $t = \ln{s/\epsilon}$
we have 
$$ h(sx,\epsilon y) = (s f_{y,s}(x), \epsilon g_{x,s}(y))$$
so as $s\to 0$ we have $g_{x,s}\to$ the identity.
For general $x,y$ we likewise get
$$ h(x,y) = (|x|f_{y/|y|,s}(x/|x|), |y| g_{x/|x|,s}(y/|y|))$$
where $s = |y||x|/\epsilon$.

To get a nontrivial example if $n>k>1$ 
the trick is to pick $f$ and $g$ so that the resulting
$h$ is smooth.
For example, we could pick a nontrivial
smooth curve
$\tau\colon (-\epsilon,\epsilon)\to O(k)$ with $\tau(0)$ the identity and let
$g_{x,s}$ be the identity and let
$f_{y,s} = \tau(e^{-1/s})$.
A tedious\footnote{
We have $h(x,y) = (\tau(e^{-1/s})(x), y)$ so for $i\le k$
we have
$\Sigma_{j=1}^k x_j\partial h_i/\partial x_j - 
\Sigma _{j=1}^{n-k} y_j \partial h_i/\partial y_j
= \tau_i(e^{-1/s})(x)  + \tau_i'(e^{-1/s})(x) e^{-1/s}s^{-2}(
\Sigma_{j=1}^k x_j^2s/|x|^2 - \Sigma _{j=1}^{n-k} y_j^2 s/|y|^2
)= \tau_i(e^{-1/s})(x)= h_i(x,y)$.
This shows the tedious part of equation (1).} calculation shows that the resulting $h$
is $C^\infty$ smooth and satisfies equations (1) and (2).


\begin{thebibliography}{9}



\bibitem{AK}
S.~Akbulut and H.~King, \emph{Topology of Real Algebraic Sets},
MSRI Publications 25, Springer Verlag (1992).









\bibitem{BZ} J-M.~Bismut and W.~Zhang, \emph{An extension of a theorem by Cheeger and M\"uller}, Astérisque No. 205 (1992).

\bibitem{BFK} D.~Burghelea, L.~Friedlander, and T.~Kappeler, \emph{On the space of trajectories of a generic gradient-like vector field},
arXiv:1101.0788v2 [math.DS]

	\bibitem{GWPL}
	C.~G.~Gibson, K.~Wirthm\"{u}ller, A.~A.~du Plessis, E.~J.~N.~Loo\ij enga,
	\emph{Topological Stability of Smooth Mappings},
	Lecture Notes in Mathematics 552, Springer-Verlag (1976).
	
	\bibitem{corners}
	D.~Joyce, \emph{On manifolds with corners}, arXiv:0910.3518v2 [math.DG]
	
\bibitem{Tico}
H.~King,
\emph{Tico Spines}
arXiv:1602.02608 [math.GT]


\bibitem{KT} H.~King and D.~Trotman, \emph{Poincare-Hopf Theorems on Singular Spaces}, 
Proceedings of the London Math.~Soc., (2014) 108 (3), pp.~682-703.

\bibitem{Mather}
J.~Mather, \emph{Notes on Topological Stability}, currently available at http://web.math.princeton.edu/facultypapers/mather/notes\_on\_topological\_stability.pdf


\bibitem{N} L.~Nicolaescu, \emph{Tame Flows}, AMS Memoirs 980.

\bibitem{Q1} L.~Qin, \emph{ On moduli spaces and CW structures arising from Morse theory on Hilbert manifolds}, J. Topol. Anal. 2 (2010), no. 4, 469–526.

\bibitem{Q2} L.~Qin, \emph{ An application of topological equivalence to Morse Theory}, arXiv:1102.2838 [math.GT]

\bibitem{Thom} R.~Thom, \emph{Ensembles et morphismes stratifi\'es}, 
Bull.~A.~M.~S., 75 (1969), pp.~240-284.


\end{thebibliography}
\end{document}